\documentclass[a4paper,12pt]{article}
\usepackage{amsmath,amssymb,amsthm,cite,verbatim,xcolor,tikz-cd}
\usepackage[bookmarks=true]{hyperref}
\usepackage{graphicx}
\usepackage{longtable}
\usepackage{lipsum}

\newcommand\blfootnote[1]{%
  \begingroup
  \renewcommand\thefootnote{}\footnote{#1}%
  \addtocounter{footnote}{-1}%
  \endgroup
}

\title{Intermediate modular curves with infinitely many cubic points over $\mathbb{Q}$}

\author{Tarun  Dalal
}
\newcounter{teo}
\newtheorem{prop}[teo]{Proposition}
\newtheorem{lema}[teo]{Lemma}
\newtheorem{thm}[teo]{Theorem}

\theoremstyle{definition}

\theoremstyle{remark}

\numberwithin{equation}{section}

\newcommand{\Q}{\mathbb{Q}}

\newcommand{\N}{\mathbb{N}}
\newcommand{\Z}{\mathbb{Z}}

\newcommand{\SL}{\operatorname{SL}}
\newcommand{\End}{\operatorname{End}}

\newcommand{\pmat}[4]{ \begin{pmatrix} #1 & #2 \\ #3 & #4 \end{pmatrix}}

\newcommand{\psmat}[4]{\bigl( \begin{smallmatrix} #1 & #2 \\ #3 & #4 \end{smallmatrix} \bigr)}

\date{}

\begin{document}

\maketitle

\begin{abstract}
In this article, we determine all intermediate modular curves $X_\Delta(N)$ that admit infinitely many cubic points over the rational field $\mathbb{Q}$.
\end{abstract}

\section{Introduction}
Let $C$ be a non singular smooth curve of genus $g_C>1$ defined over a number field $K$.
 A point $P\in C$ is said to
be a point of degree $d$ over $K$  if $[K(P):K]=d.$
We consider the set of all points of degree $d$ on $C$ by
$\Gamma_d'(C,K)=\cup_{[L:K]=d}C(L)$, where $L\subseteq \overline{K}$
runs over the finite extensions of $K$ and $C(L)$ denotes the set of all $L$-rational points on $C$ (here we fix once and for all
$\overline{K}$, an algebraic closure of $K$). 
We say the curve $C$ has infinitely many points of degree $d$ over $K$ if and only if the set $\Gamma^\prime_d(C,K)$ is infinite.
By Faltings' theorem, we know that $C(L)$ is finite for any number field $L$.
Now it is natural to ask for $d\geq 2$, whether the set $\Gamma_d'(C,K)$ is finite or not. In this article, we are interested in investigating this question for intermediate modular curves lying between $X_1(N)$ and $X_0(N)$. We now set up some notations.

For any $N\in \N$, we define the groups
$$\Gamma_1(N):= \Big\{\pmat{a}{b}{c}{d}\in \SL_2(\Z): a\equiv d \equiv 1 \pmod N, \ c\equiv 0 \pmod N \Big\},$$
$$\Gamma_0(N):= \Big\{\pmat{a}{b}{c}{d}\in \SL_2(\Z):  c\equiv 0 \pmod N \Big\}.$$

For any subgroup $\Delta\subseteq (\Z/N \Z)^\times$, we define the group
$$\Gamma_\Delta(N):= \Big\{\pmat{a}{b}{c}{d}\in \SL_2(\Z): c\equiv 0 \pmod N, \ (a \ \mathrm{mod} \ N)\in \Delta \Big\}.$$
Let $X_1(N), X_0(N)$ and $X_\Delta(N)$ be the modular curves corresponding to the groups $\Gamma_1(N), \Gamma_0(N)$ and $\Gamma_\Delta(N)$ respectively. 
Since $-\mathrm{Id}$ acts as identity on the complex upper half plane, we have $X_{\Delta}(N)=X_{\langle \pm 1, \Delta \rangle}(N)$. Thus we will always assume that $-1\in \Delta$.
Observe that for any natural number $N$ and any subgroup $\Delta\subseteq (\Z/NZ)^\times$, 
the inclusions $\Gamma_1(N)\subseteq \Gamma_\Delta(N)\subseteq \Gamma_0(N)$ induces the natural $\Q$-rational mappings:
$X_1(N)\rightarrow X_{\Delta}(N)\rightarrow X_0(N).$

Furthermore, for $N \in \N$, the operator $w_N:=\psmat{0}{-1}{N}{0}$ act as an involution on $X_0(N)$. Therefore there is a degree $2$ mapping $X_0(N)\rightarrow X_0(N)/w_N=:X_0^+(N)$. Consequently, we get the mappings
$$X_1(N)\rightarrow X_{\Delta}(N)\rightarrow X_0(N)\rightarrow X_0^+(N).$$

It is well known that all these modular curves have models defined over $\Q$. For these modular curves, the set $\Gamma_d^\prime(\textendash,\Q)$ have been studied by many authors. For $d=2,3$ we know the values of $N$ for which the set $\Gamma_d^\prime(C,\Q)$ is infinite, where $C$ is one of the modular curves $X_1(N), X_0(N)$ and $X_0^+(N)$ (cf. \cite{KM95}, \cite{JKS04}, \cite{B99}, \cite{Jeo21}, \cite{Jeo18}, \cite{BD22}).

On the other hand, for the modular curves $X_{\Delta}(N)$, whether the set $\Gamma_d^\prime(X_{\Delta}(N), \Q)$ is finite or not is known only for the case $d=2$. In \cite{JKS20}, Jeon, Kim and Schweizer determined the values of $N$ and $\Delta$ for which the set $\Gamma_2^\prime(X_\Delta(N), \Q)$ is infinite. Now it is natural to ask for what values of $N$ and $\Delta$ the set $\Gamma_3^\prime(X_{\Delta}(N), \Q)$ is infinite.

In this article, we answer this question i.e. we determine the values of $N$ and $\Delta$ for which the set $\Gamma_3^\prime(X_\Delta(N), \Q)$ is infinite. 
Since $X_\Delta(N)=X_1(N)$ (resp., $X_\Delta(N)=X_0(N)$) when $\Delta=\{\pm 1\}$ (resp., $\Delta=(\Z/N\Z)^\times$), 
we always assume that $\{\pm 1\}\subsetneq \Delta \subsetneq (\Z/N\Z)^\times$. 

The Atkin-Lehner involutions play an important role in determining whether the sets $\Gamma_3^\prime(X_0(N), \Q)$ and $\Gamma_3^\prime(X_0^+(N), \Q)$ are finite or not. Unfortunately, the Atkin-Lehner operators may not define an involution on $X_{\Delta}(N)$ (moreover, they may not be defined over $\Q$) (cf. \cite[\S 1]{JKS20}). Thus the existing arguments for $X_1(N), X_0(N)$ and $X_0^+(N)$ are not sufficient to determine whether the set $\Gamma_3^+(X_{\Delta}(N), \Q)$ is finite or not. 
To deal with the curves $X_{\Delta}(N)$ for which the existing argument will not work, we use a result of Glenn Stevens (cf. \cite[Proposition 1.4]{Ste89}) and study the ramification of certain mapping.
Our main theorem of this article is the following:
\begin{thm}\label{MT}
The set $\Gamma^\prime_3(X_\Delta(N), \Q)$ is infinite if and only if $g_{X_\Delta(N)}\leq 1$ (cf. Table \ref{Gon leq 3 table} for the values of $N,\Delta$ with $g_{X_\Delta(N)}\leq 1$ ) or  $N$ and $\Delta$ are in the following list:
\begin{center}
\begin{tabular}{|c|c|c|}
\hline
$N$ & $\{\pm 1\} \subsetneq \Delta \subsetneq (\Z/N\Z)^\times$ &
$g_{X_\Delta(N)}$
 \\ \hline
  $24$ & $\Delta_1=\{\pm 1,\pm 5\}$ & $3$
 \\ \hline
 $24$ & $\Delta_2=\{\pm 1,\pm 7\}$ & $3$
 \\ \hline
   $26$ & $\Delta_1=\{\pm 1,\pm 5\}$ & $4$
 \\ \hline
 $26$ & $\Delta_2=\{\pm 1,\pm 3,\pm 9\}$ & $4$
 \\ \hline
  $28$ & $\Delta_1=\{\pm 1,\pm 13\}$ & $4$
 \\ \hline
 $28$ & $\Delta_2=\{\pm 1,\pm 3,\pm 9\}$ & $4$
 \\ \hline
 $29$ & $\Delta_2=\{\pm 1,\pm 4,\pm 5,\pm 6,\pm 7,\pm 9,\pm 13\}$ & $4$
 \\ \hline
  $36$ & $\Delta_2=\{\pm 1,\pm 11,\pm 13\}$ & $3$
 \\ \hline
  $37$ & $\Delta_3=\{\pm 1,\pm 6,\pm 8,\pm 10,\pm 11,\pm 14\}$ & $4$
 \\ \hline
 $37$ & $\Delta_4=\{\pm 1,\pm 3,\pm 4,\pm 7,\pm 9,\pm 10,\pm 11,\pm 12,\pm 16\} $ & $4$
 \\ \hline
  $49$ & $\Delta_2=\{\pm 1,\pm 6,\pm 8,\pm 13,\pm 15,\pm 20,\pm 22\}$ & $3$
 \\ \hline
 $50$ & $\Delta_2=\{\pm 1,\pm 9,\pm 11,\pm 19,\pm 21\}$ & $4$
 \\ \hline
\end{tabular}
\end{center}
 
\end{thm}
All the source codes to verify the computations in this paper are available at \url{https://github.com/Tarundalalmath/Intermediate_modular_curve}.

\section{Preliminary}

Since for any $N\in \N$ and any subgroup $\Delta\subseteq (\Z/N\Z)^\times$ there is a natural $\Q$-rational mapping $X_\Delta(N) \rightarrow X_0(N)$, it is easy to see that if the set $\Gamma_3^\prime(X_\Delta(N), \Q)$ is infinite, then the set $\Gamma_3^\prime(X_0(N), \Q)$ is also infinite.
 From \cite[Theorem 0.1]{Jeo21}, we know that the set $\Gamma_3^\prime(X_0(N), \Q)$ is infinite if and only if 
$$N \in S:= \{1-29, 31, 32, 34, 36, 37, 43, 45, 49, 50, 54, 64, 81\}.$$
Therefore the set $\Gamma_3^\prime(X_\Delta(N), \Q)$ is finite if $N \notin S$ (for all values of $\Delta$). Thus it is enough to investigate the values $N\in S$.
Henceforth, we will always assume that $N\in S$.
Furthermore, since we are assuming $-1\in \Delta$, it is enough to consider the values of $N, \Delta$ appearing in Table \ref{Gon leq 3 table}, Table \ref{Gon =3 table}, Table \ref{Gon >3 table and bielliptic} and Table \ref{Gon >3 table}.

For a complete curve $C$ over $K$, the gonality of $C$ is defined as follows:
$$\mathrm{Gon}(C):=\mathrm{min}\{\deg(\varphi)\mid \varphi: C\rightarrow \mathbb{P}^1 \mathrm{defined \ over} \ \overline{K}\}.$$
Observe that, if there is a $K$-rational degree $3$ mapping from $C$ to the projective line $\mathbb{P}^1$ or to an elliptic curve $E$ with positive $K$-rank, then the set $\Gamma^\prime_3(C,K)$ is infinite. In \cite{Jeo21}, the author proved the following converse part of this statement.


\begin{lema} \label{N not in known set the it has map to elliptic curve}\cite[Lemma 1.2]{Jeo21}(see also \cite[Lemma 2]{BD22})
Suppose $Gon(C)\geq 4$, { $C(K)\ne \emptyset$} and $C$ does not have a degree $\leq 2$ map to
an elliptic curve. If the set $\Gamma_3'(C,K)$ is infinite then $C$ admits a $K$-rational map of degree $3$ to an elliptic
curve with positive $K$-rank. 
In particular, the Jacobian variety $Jac(C)(K)$ has positive rank.
\end{lema}
For $N\in S$, the following results are known:
\begin{thm}
\begin{enumerate}
\item (\cite{IM91}, \cite{JK07}) The values of $N,\Delta$ for which $\mathrm{Gon}(X_\Delta(N))\leq  2$ are given in Table \ref{Gon leq 3 table}. Moreover,  the curve $X_{\Delta_1}(21)$ is the only hyperelliptic curve.
\item(\cite{JK07}) The values of $N,\Delta$ for which $\mathrm{Gon}(X_\Delta(N))= 3$ are given in Table \ref{Gon =3 table}.
\item(\cite{JKS20}) The values of $N,\Delta$ for which $\mathrm{Gon}(X_\Delta(N))> 3$ and $X_\Delta(N)$ is bielliptic are given in Table \ref{Gon >3 table and bielliptic} (a curve is said to be bielliptic if it has a degree $2$ mapping to an elliptic curve).
\end{enumerate}

\end{thm}

The values of $N,\Delta$ for which $\mathrm{Gon}(X_\Delta(N))> 3$ and $X_\Delta(N)$ does not have a degree $\leq$ 2 map to an elliptic curve are given in Table \ref{Gon >3 table}. Since the modular curve $X_\Delta(N)$ always has a $\Q$-rational point (cusp) lies above the cusp $0$ of $X_0(N)$ (cf. \cite[Lemma 1.2]{IM91} or \cite[Lemma 1.1]{JKS20} for a more general statement), we can apply Lemma \ref{N not in known set the it has map to elliptic curve} for the values of $N$ and $\Delta$ in Table \ref{Gon >3 table}.

\section{Proof of Theorem \ref{MT}}
In this section, we investigate the values of $N,\Delta$ appearing in Table \ref{Gon leq 3 table} -Table \ref{Gon >3 table} and prove Theorem \ref{MT}. 
We start with the values of $N$ and $\Delta$ given in Table \ref{Gon leq 3 table}.
\subsection{$N,\Delta$ in Table \ref{Gon leq 3 table}}


The modular curves $X_{\Delta}(N)$ for $\Delta$ and $N$ in Table \ref{Gon leq 3 table} have gonality $\leq 2$. 
Since the curve $X_{\Delta}(N)$ has always has a $\Q$-rational point (cusp), it is easy to see that $X_{\Delta}(N)$ has a degree $3$ mapping defined over $\Q$ to the projective line $\mathbb{P}^1$ when $g_{X_\Delta(N)}\leq 1$. Thus the set $\Gamma_3^\prime(X_{\Delta}(N), \Q)$ is infinite when $g_{X_{\Delta}(N)}\leq 1$.


The only remaining case in Table \ref{Gon leq 3 table} is the curve $X_{\Delta_1}(21)$. 
Note that $X_{\Delta_1}(21)$ is the only hyperelliptic curve (cf. \cite[Theorem 2.3]{JK07}) and $g_{X_{\Delta_1}(21)}\geq 3$. Since the Jacobian $Jac(X_1(21))(\Q)$ has rank zero (cf. \cite[Theorem 3.1]{sporadic}), the Jacobian $Jac(X_{\Delta_1}(21))(\Q)$ also has rank zero.
Since $X_{\Delta_1}(21)$ is hyperelliptic curve (cf. \cite[Theorem 2.3]{JK07}) of genus $\geq 3$, a similar argument as in \cite[Lemma 2.6 and Lemma 2.7]{Jeo21} show that if $\Gamma_3^\prime(X_{\Delta_1}(21), \Q)$ is infinite, then the Jacobian $Jac(X_{\Delta_1}(21))(\Q)$ has positive rank. 
This is a contradiction, and thus we obtain:

\begin{thm}
The set $\Gamma_3^\prime(X_{\Delta_1}(21), \Q)$ is finite.
\end{thm}

\subsection{$N, \Delta$ in Table \ref{Gon =3 table}}
We now deal with the values of $N$ and $\Delta$ appearing in Table \ref{Gon =3 table}. The modular curves $X_{\Delta}(N)$ for $\Delta$ and $N$ in Table \ref{Gon =3 table} have gonality equal to $3$. 
First consider the cases with $g_{X_\Delta(N)}=3$. In these cases $X_\Delta(N)$ is trigonal (i.e has a degree $3$ mapping to the projective line) and 
the projection from the $\Q$-rational cusp defines a degree $3$ mapping $X_\Delta(N)\rightarrow \mathbb{P}^1$ over $\Q$ (cf. \cite[p. 136]{HaSh99a}). Hence we obtain:
\begin{thm}
All the sets $\Gamma_3^\prime(X_{\Delta_1}(24), \Q), \Gamma_3^\prime(X_{\Delta_2}(24), \Q), \Gamma_3^\prime(X_{\Delta_2}(36), \Q)$ and $\Gamma_3^\prime(X_{\Delta_2}(49), \Q)$ are infinite.
\end{thm}
We now consider $N$ and $\Delta$ such that $g_{X_\Delta(N)}=4$ (cf. Table \ref{quadratic surface for genus 4 curves}). In these cases $X_{\Delta}(N)$ is non hyperelliptic and $\mathrm{Gon}(X_{\Delta}(N))=3$, i.e. there is degree $3$ mapping $X_{\Delta}(N)\rightarrow \mathbb{P}^1$ defined over $\overline{\Q}$. 
Using Petri's theorem, we can construct a model for $X_{\Delta}(N)$ over $\Q$.
It is well known that a non hyperelliptic curve $C$ (defined over $\Q$) of genus $4$ lies either on a ruled surface or on a quadratic cone (defined over either $\Q$, a quadratic field or a biquadratic field) (cf. \cite[Page 131]{HaSh99a}). If the ruled surface or the quadratic cone is defined over $\Q$, then $C$ is trigonal over $\Q$ i.e. there is a degree $3$ mapping $C\rightarrow \mathbb{P}^1$ defined over $\Q$. For example, consider the curve $X_{\Delta_1}(26)$. Using \texttt{SAGE}, we obtain the following basis of weight $2$ cusp form $S_2(\Gamma_{\Delta_1}(26))$:
\begin{align*}
f_1:=&q - 2q^5 - q^6 - 3q^7 + 2q^9 + q^{10} + O(q^{11}),\\
f_2:=&q^2 - q^5 - q^6 - 3q^7 - q^8 + 2q^9 + O(q^{11}),\\
f_3:=&q^3 - q^5 - 2q^7 - q^8 - q^9 + q^{10} + O(q^{11}),\\
f_4:=&q^4 - q^5 - q^6 + q^7 - q^8 - q^9 + q^{10} + O(q^{11}).
\end{align*}
Then using this basis, from \texttt{MAGMA} we see that a model of $X_{\Delta_1}(26)$ is given by (we refer the readers to \cite{Ga} for a detailed discussion on how to construct the models):
$$X_{\Delta_1}(26):=
\begin{cases}
x^2 z - x y^2 - x z^2 + 2 y^2 z - 2 y z^2 + y z w - y w^2 + z^3 - 2 z^2 w + z w^2,\\
x w - y z + z^2.
\end{cases}
$$
After some suitable co-ordinate changes the quadratic form $x w - y z + z^2$ can be written as 
$$x^2 + y^2 - z^2 - w^2= (x+z)(x-z)-(w+y)(w-y),$$
which is isomorphic to the ruled surface $uv-st$ defined over $\Q$. Hence there is a degree $3$ mapping $X_{\Delta_1}(26)\rightarrow \mathbb{P}^1$ defined over $\Q$. Thus the set $\Gamma_3^\prime(X_{\Delta_1}(26), \Q)$ is infinite.

Similar arguments show that the curves $X_{\Delta_2}(26), X_{\Delta_1}(28), X_{\Delta_2}(29), X_{\Delta_3}(37), X_{\Delta_2}(50)$ lie on ruled surfaces over $\Q$ and the curves $X_{\Delta_2}(28), X_{\Delta_4}(37)$ lie on quadratic cones over $\Q$ (cf. Table \ref{quadratic surface for genus 4 curves}). Hence all these curves are trigonal over $\Q$.
Consequently, we obtain:
\begin{thm}
The set $\Gamma_3^\prime(X_{\Delta}(N), \Q)$ is infinite for the following pairs of $N$ and $\Delta$:
$$(26, \Delta_1), (26, \Delta_2), (28, \Delta_1), (28, \Delta_2), (29, \Delta_2), (37, \Delta_3), (37, \Delta_4), (50, \Delta_2).$$
\end{thm}
On the other hand, the curve $X_{\Delta_1}(25)$ lies on a ruled surface defined over $\Q(\sqrt{5})$ (cf. Table \ref{quadratic surface for genus 4 curves}). Therefore $Gon(X_{\Delta_1}(25))=3$ but 
there is no degree $3$ mapping $X_{\Delta_1}(25)\rightarrow \mathbb{P}^1$ defined over $\Q$.
A similar argument as in \cite[p. 352]{Jeo21} shows that if the set $\Gamma_3^\prime(X_{\Delta_1}(25), \Q)$ is infinite, then the Jacobian $Jac(X_{\Delta_1}(25))(\Q)$ has positive rank. In particular, this implies that $Jac(X_1(25))(\Q)$ has positive rank. Since $Jac(X_1(25))(\Q)$ has rank zero (cf. \cite[Theorem 3.1]{sporadic}), we conclude that
\begin{thm}
The set $\Gamma_3^\prime(X_{\Delta_1}(25), \Q)$ is finite.
\end{thm}
\subsection{$N,\Delta$ in Table \ref{Gon >3 table and bielliptic}}

For $N$ and $\Delta$ in Table \ref{Gon >3 table and bielliptic}, the modular curves $X_{\Delta}(N)$ are bielliptic and $\mathrm{Gon}(X_{\Delta}(N))>3$. By \cite[Theorem 1.2]{JKS04}, if the set $\Gamma_3^\prime(X_{\Delta}(N), \Q)$ is infinite, then the Jacobian $Jac(X_{\Delta}(N))$ contains an elliptic curve with positive $\Q$-rank. 
In particular, the Jacobian $Jac(X_{\Delta}(N))(\Q)$ has positive rank. Which forces that the Jacobian $Jac(X_{1}(N))(\Q)$ has positive rank.
Since $Jac(X_{1}(N))(\Q)$ has rank zero for $N \in \{32, 34, 45, 64\}$ (cf. \cite[Theorem 3.1]{sporadic}), we conclude the following result:
\begin{thm}
The sets $\Gamma_3^\prime(X_{\Delta_1}(32), \Q), \Gamma_3^\prime(X_{\Delta_2}(34), \Q), \Gamma_3^\prime(X_{\Delta_4}(45), \Q)$  and $\Gamma_3^\prime(X_{\Delta_3}(64), \Q)$ are finite.
\end{thm}

\subsection{$N, \Delta$ in Table \ref{Gon >3 table}}
Finally, we deal with the values of $N$ and $\Delta$ in Table \ref{Gon >3 table}.
The modular curves $X_{\Delta}(N)$ for $\Delta$ and $N$ in Table \ref{Gon >3 table} are not bielliptic and $\mathrm{Gon}(X_{\Delta}(N))>3$. 
For $N$ and $\Delta$ in Table \ref{Gon >3 table}, if the set $\Gamma^\prime_3(X_\Delta(N), \Q)$ is infinite, then by Lemma \ref{N not in known set the it has map to elliptic curve}, there exists a $\Q$-rational degree $3$ mapping $X_\Delta(N)\rightarrow E$, where $E$ is an elliptic curve with positive $\Q$-rank. This induces a mapping $X_1(N)\rightarrow E$.
Thus $\mathrm{Cond}(E)$ divides $N$. 
Observe that for $N\in \{29,31,34,36,45,49,50,54,64,81\}$ there is no elliptic curve $E$ with positive $\Q$-rank and $\mathrm{Cond}(E)|N$ (moreover for these values of $N$, the Jacobian $Jac(X_1(N))(\Q)$ has rank $0$ (cf. \cite[Theorem 3.1]{sporadic})).
As an immediate consequence, we obtain the following result.
\begin{thm}
The set $\Gamma_3^\prime(X_\Delta(N), \Q)$ is finite for the following pairs of $N$ and $\Delta$:

$(29, \Delta_1), (31, \Delta_1), (31, \Delta_2), (34, \Delta_1), (36, \Delta_1), (45, \Delta_1), (45, \Delta_2), (45, \Delta_3), (49, \Delta_1), (50, \Delta_1),$ 

$(54, \Delta_1), (64, \Delta_1), (64, \Delta_2), (81, \Delta_1), (81, \Delta_2)$.

%
\end{thm}

To deal with the remaining cases in Table \ref{Gon >3 table}, we recall the following proposition.

\begin{prop}\cite[Proposition 1.4]{Ste89}
\label{optimal curve for X1}
Let $\mathcal{A}$ be an isogeny class (over $\Q$) of modular elliptic curves of level $N$. Then there is a curve $E_1\in \mathcal{A}$ and a modular parametrization
$$\pi_1: X_1(N)\rightarrow E_1$$
such that: if $\pi: X_1(N)\rightarrow E$ is a parametrization of a curve $E\in \mathcal{A}$, then there is an isogeny $\beta : A_1 \rightarrow A$ which makes the following diagram commutative:
$$
\begin{tikzcd}
&X_1(N)\arrow[r, "\pi_1"] \arrow[d,"\pi"] & E_1 \arrow[dl,"\beta"]\\
&E
\end{tikzcd}.
$$ 
The curve $E_1$ is called the optimal curve for $X_1(N)$, and the mapping $\pi_1$ is called the optimal modular parametrization for $E_1$ and $X_1(N)$.
\end{prop}
We now apply Proposition \ref{optimal curve for X1} to get the following result.
\begin{thm}
The sets $\Gamma_3^\prime(X_{\Delta_1}(37), \Q), \Gamma_3^\prime(X_{\Delta_1}(43),\Q)$ and $\Gamma_3^\prime(X_{\Delta_2}(43), \Q)$ are finite.
\end{thm}
\begin{proof}
First we consider the curve $X_{\Delta_1}(37)$. Note that $37a1$ is the only elliptic curve with positive $\Q$-rank whose conductor divides $37$. If the set $\Gamma_3^\prime(X_{\Delta_1}(37), \Q)$  is infinite, then by Lemma \ref{N not in known set the it has map to elliptic curve} there exists a $\Q$-rational degree $3$ mapping 
$f:X_{\Delta_1}(37)\rightarrow 37a1.$

Consequently, we get a degree $6$ mapping 
$$\begin{tikzcd}
&\varphi: X_1(37)\arrow[r, "\phi"', "\deg  2"''] &X_{\Delta_1}(37)\arrow[r, "f"', "\deg  3"''] & 37a1 
\end{tikzcd}.
$$ 

Since $37a1$ is the only elliptic curve in its isogeny class,
it is the optimal curve for $X_1(37)$. Let $\pi_1: X_1(37)\rightarrow 37a1$ be the optimal modular parametrization.
By Proposition \ref{optimal curve for X1}, there exists an isogeny $\beta_1:37a1\rightarrow 37a1$ such that the following diagram commutes
\begin{equation}
\label{1st commutative diagram for 37 delta 1}
\begin{tikzcd}
X_1(37) \arrow[d, "\phi", "\deg 2"'] \arrow[r, "\pi_1"]  & 37a1\arrow[ddl, "\beta_1"]\\
X_{\Delta_1}(37) \arrow[d, "f", "\deg 3"']\\
37a1
\end{tikzcd}
\end{equation}
Since $\End(37a1)\cong \Z$, the isogeny $\beta_1$ is a multiplication by $m$-mapping (for some $m\in \N$). Consequently, $\deg(\beta_1)$ is a square.
By considering the degrees of the mappings in \eqref{1st commutative diagram for 37 delta 1}, we get $\deg (\beta_1) = \frac{6}{\deg (\pi_1)}$.  Which forces that $\deg(\beta_1)=1$, i.e $\beta_1$ is an isomorphism.
Hence without loss of generality, we can assume that the mapping   
$$\begin{tikzcd}
&\varphi: X_1(37)\arrow[r, "\phi"', "\deg  2"''] &X_{\Delta_1}(37)\arrow[r, "\deg 3","f"'] & 37a1 
\end{tikzcd}
$$ 
is the optimal modular parametrization for $37a1$ and $X_1(37)$.

{
On the other hand, we know that $X_0^+(37)\cong 37a1$, and there is a $\Q$-rational mapping $g_1:X_{\Delta_1}(37)\rightarrow X_0^+(37)$ of degree $18$, which is the composition of the natural projection maps  
\begin{equation*}
\begin{tikzcd}
&X_{\Delta_1}(37)\arrow[r, "\deg  9"''] &X_{0}(37)\arrow[r, "w_{37}"', "\deg  2"''] & X_0^+(37)\cong 37a1 
\end{tikzcd}. 
\end{equation*}
Hence we obtain the mappings
$$\begin{tikzcd}
&X_1(37)\arrow[r,"\phi"', "\deg  2"''] &X_{\Delta_1}(37)\arrow[r, "g_1"', "\deg 18"''] & X_0^+(37)\cong 37a1 
\end{tikzcd}.
$$


Since the mapping $\varphi=f\circ \phi$ is the optimal modular parametrization for $37a1$ and $X_1(37)$, by Proposition \ref{optimal curve for X1} there exists an isogeny $\beta:37a1\rightarrow 37a1$ such that the following diagram commutes
\begin{equation}
\label{commutative diagram for 37 delta 1}
\begin{tikzcd}
&X_1(37)\arrow[r, "\phi"] &X_{\Delta_1}(37)\arrow[r, "f"] \arrow[d,"g_1"] & 37a1 \arrow[dl,"\beta"]\\
& &X_0^+(37)\cong 37a1
\end{tikzcd}.
\end{equation}
Since $\End(37a1)\cong \Z$, the isogeny $\beta$ is a multiplication by $m$-mapping (for some $m\in \N$). Consequently, $\deg(\beta)$ is a square. On the other hand, 
considering the degrees of the mappings in \eqref{commutative diagram for 37 delta 1}, we get $\deg(\beta)=6$, which is not a square. This is a contradiction.
Hence there is no $\Q$-rational degree $3$ mapping between $X_{\Delta_1}(37)$ and $37a1$.
Therefore the set $\Gamma_3^\prime(X_{\Delta_1}(37), \Q)$ is finite.
}

A similar argument will work for the cases $X_{\Delta_1}(43)$ and $X_{\Delta_2}(43)$ (note that the mappings $X_{\Delta_1}(43)\rightarrow X_0^+(43)$ and $X_{\Delta_2}(43)\rightarrow X_0^+(43)$ has degree $14$ and $6$ respectively). The details are left to the reader.
\end{proof}

\subsubsection{The curve $X_{\Delta_2}(37)$}
\label{37 delta 2 section}
After all the above discussions, we are only left to check the curve $X_{\Delta_2}(37)$. 
Note that $37a1$ is the only elliptic curve with positive $\Q$-rank and conductor divides $37$. 
By Lemma \ref{N not in known set the it has map to elliptic curve}, if $X_{\Delta_2}(37)$ has infinitely many cubic points over $\Q$, then there exists a $\Q$-rational degree $3$ mapping 
$$f:X_{\Delta_2}(37)\rightarrow 37a1.$$
Hence we get a degree $9$ mapping 
$$\begin{tikzcd}
&\varphi: X_1(37)\arrow[r, "\phi"', "\deg   3"''] &X_{\Delta_2}(37)\arrow[r, "f"', "\deg  3"''] & 37a1 
\end{tikzcd}.
$$ 

Since $37a1$ is the only elliptic curve in its isogeny class,
it is the optimal curve for $X_1(37)$. Let $\pi_1: X_1(37)\rightarrow 37a1$ be the optimal modular parametrization.
By Proposition \ref{optimal curve for X1}, there exists an isogeny $\beta_1:37a1\rightarrow 37a1$ such that the following diagram commutes.
$$\begin{tikzcd}
X_1(37) \arrow[d, "\deg 3"'] \arrow[r, "\pi_1"]  & 37a1\arrow[ddl, "\beta_1"]\\
X_{\Delta_2}(37) \arrow[d, "\deg 3"']\\
37a1
\end{tikzcd}
$$
By considering the degrees of the above mappings, we have $\deg (\beta_1) = \frac{9}{\deg (\pi_1)}$. Since $\End(37a1)\cong \Z$, the isogeny $\beta_1$ is a multiplication by $m$-mapping, for some $m\in \N$. Consequently, $\deg(\beta_1)$ is a square. Since $\pi_1$ is not an isomorphism, we conclude that $\beta_1$ is an isomorphism (i.e $\deg(\beta)=1$).
Hence without loss of generality, we can assume that $\pi_1=\varphi$, i.e. the mapping   
$$\begin{tikzcd}
&\varphi: X_1(37)\arrow[r, "\phi"', "\deg  3"''] &X_{\Delta_2}(37)\arrow[r, "f"] & 37a1 
\end{tikzcd}
$$ 
is the optimal modular parametrization for $37a1$ and $X_1(37)$.

On the other hand, we know that $X_0^+(37)\cong 37a1$, and there is a natural $\Q$-rational mapping $g:X_{\Delta_2}(37)\rightarrow X_0^+(37)$ of degree $12$, which is the composition of the natural projection maps  
\begin{equation}
\begin{tikzcd}
&X_{\Delta_2}(37)\arrow[r,"\pi"', "\deg  6"''] &X_{0}(37)\arrow[r, "w_{37}"', "\deg  2"''] & X_0^+(37)\cong 37a1 
\end{tikzcd}. 
\end{equation}
Consequently, we get the mappings 
$$\begin{tikzcd}
&X_1(37)\arrow[r,"\phi"', "\deg  3"''] &X_{\Delta_2}(37)\arrow[r, "g"] & X_0^+(37)\cong 37a1 
\end{tikzcd}.
$$


By Proposition \ref{optimal curve for X1}, there exists an isogeny $\alpha:37a1\rightarrow 37a1$ such that the following diagram commutes.

\begin{equation}
\label{final commutative diagram for 37}
\begin{tikzcd}
&X_1(37)\arrow[r, "\deg  3"'', "\phi"']&X_{\Delta_2}(37) \arrow[d, "\deg 6"', "\pi"] \arrow[r, "\deg 3"'', "f"']  & 37a1\arrow[ddl, "\alpha"]\\
&&X_{0}(37) \arrow[d, "\deg 2"', "w_{37}"]\\
&&X_0^+(37)
\end{tikzcd}
\end{equation}


Since $\mathrm{End}(37a1)\cong \Z$, considering the degrees of the mappings in \eqref{final commutative diagram for 37} we conclude that the isogeny $\alpha$ is the multiplication by $2$-mapping. We now study the ramifications of the mappings in \eqref{final commutative diagram for 37}.

First consider the mapping $g=\pi \circ w_{37}: X_{\Delta_2}(37)\rightarrow X_0^+(37)\cong 37a1$. 
The degree $2$ mapping $w_{37}: X_0(37)\rightarrow X_0^+(37)$ has two distinct fixed points (cf. \cite[Page 27]{MaSD}), say $w_1$ and $w_2$.
Let $\bar{w}_1$ and $\bar{w}_2$ denote the images of $w_1$ and $w_2$, respectively, under the mapping $w_{37}$.
The points of $X_{\Delta_2}(37)$ lying above the points $w_1$ and $w_2$ do not ramify in the mapping $\pi: X_{\Delta_2}(37)\rightarrow X_0(37)$ (cf. \cite[Lemma 2.4]{JKS20}).
Thus all the points of $X_{\Delta_2}(37)$ lying above the points $\bar{w}_1$ and $\bar{w}_2$ have ramification index $2$ in the mapping $g$. Therefore, by ramification formula (cf. \cite[Proposition 2.6]{Sil09}), there exist exactly six points $P_1, P_2,\ldots,P_6\in X_{\Delta_2}(37)$ such that 
$g(P_i)=\bar{w}_1, \ \mathrm{for} \ 1\leq i \leq 6.$
Note that $e_g(P_i)=2$ for $1\leq i \leq 6$, where $e_g(P_i)$ denotes the ramification index of $P_i$ in the mapping $g$.

%

On the other hand, from the commutative diagram \eqref{final commutative diagram for 37}, we have $g=\alpha\circ f$.
Since $\alpha$ is the multiplication by $2$-mapping, it is unramified and $\mathrm{degree}(\alpha)=4$.
By ramification formula, there exist exactly four points $x_1, x_2, x_3,x_4\in 37a1$ such that $\alpha(x_i)=\bar{w}_1$. Note that $e_\alpha(x_i)=1$ for $1\leq i \leq 4$, where $e_\alpha(x_i)$ denotes the ramification index of $x_i$ in the mapping $\alpha$. For the mapping $f$, by ramification formula, we get
\begin{equation}
\label{ramification formula for x1}
\sum_{P\in f^{-1}(x_1)}e_f(P)=3,
\end{equation}
where $e_f(P)$ denotes the ramification index of $P$ in the mapping $f$.
Since $g=\alpha\circ f$, we have $f^{-1}(x_1)\subseteq \{P_1, P_2, \ldots, P_6\}$. Consequently \eqref{ramification formula for x1} implies that there exists $P_j$ for some $1\leq j \leq 6$ such that $f(P_j)=x_1$ and $e_f(P_j)\neq 2$. Therefore $e_g(P_j)=e_f(P_j)e_\alpha(x_1)=e_f(P_j)\ne 2$.
This is a contradiction since $e_g(P_i)=2$ for $1\leq i \leq 6$.

Hence such mapping $f$ does not exist and we conclude that:
\begin{thm}
\label{Delta 2 37}
The set $\Gamma_3^\prime(X_{\Delta_2}(37), \Q)$ is finite.
\end{thm}

\section*{Acknowledgments}
The author is deeply indebted to Prof. Francesc Bars for his comments and several helpful discussions.
The author thanks the anonymous referee for providing many valuable comments.
The author also thanks the Dept. of Atomic Energy, Govt of India for the financial support provided to carry out this research work at Harish-Chandra Research Institute.

\section{Appendix}
\begin{longtable}{|c|c|c|}
\caption{$X_\Delta(N)$ with $\mathrm{Gon}(X_\Delta(N))\leq 2$}\label{Gon leq 3 table}
\\ \hline
$N$ & $\{\pm 1\} \subsetneq \Delta \subsetneq (\Z/N\Z)^\times$ &
$g_{X_\Delta(N)}$
 \\ \hline
 $1\le N \le 12$ & No non trivial choice for $\Delta$ & $-$
 \\ \hline
  $13$ & $\Delta_1=\{ \pm 1, \pm 5\}$ & 0
 \\ \hline
 $13$ & $\Delta_2=\{ \pm 1, \pm 3, \pm 4 \}$ & 0
 \\ \hline
 $14$ & No non trivial choice for $\Delta$ & $-$
 \\ \hline
 $15$ & $\Delta_1=\{ \pm 1, \pm 4\}$ & 1
 \\ \hline
 $16$ & $\Delta_1=\{ \pm 1, \pm 7\}$ & 0
 \\ \hline
 $17$ & $\Delta_1=\{\pm 1,\pm 4\}$ & $1$
 \\ \hline
 $17$ & $\Delta_2=\{\pm 1,\pm 2,\pm 4,\pm 8\}$ & $1$
 \\ \hline
 $18$ & No non trivial choice for $\Delta$ & $-$
 \\ \hline
 $19$ & $\Delta_1=\{\pm 1,\pm 7,\pm 8\}$ & $1$
 \\ \hline
  $20$ & $\Delta_1=\{ \pm 1, \pm 9\}$ & 1
 \\ \hline
 $21$ & $\Delta_1=\{\pm 1,\pm 8\}$ & $3$
 \\ \hline
 $21$ & $\Delta_2=\{\pm 1,\pm 4,\pm 5\}$ & $1$
 \\ \hline
 $22$ & No non trivial choice for $\Delta$ & $-$
 \\ \hline
 $23$ & No non trivial choice for $\Delta$ & $-$
 \\ \hline
 $24$ & $\Delta_3=\{\pm 1,\pm 11\}$ & $1$
 \\ \hline
 $25$ & $\Delta_2=\{\pm 1,\pm 4,\pm 6,\pm 9,\pm 11\}$ & $0$
 \\ \hline
 $27$ & $\Delta_1=\{\pm 1,\pm 8,\pm 10\}$ & $1$
 \\ \hline
 $32$ & $\Delta_2=\{\pm 1,\pm 7,\pm 9,\pm 15\}$ & $1$
 \\ \hline

\end{longtable}
\begin{longtable}{|c|c|c|}
\caption{$X_\Delta(N)$ with $\mathrm{Gon}(X_\Delta(N))= 3$}\label{Gon =3 table}
 \\ \hline
 $N$ & $\{\pm 1\} \subsetneq \Delta \subsetneq (\Z/N\Z)^\times$ &
$g_{X_\Delta(N)}$
\\ \hline
 $24$ & $\Delta_1=\{\pm 1,\pm 5\}$ & $3$
 \\ \hline
 $24$ & $\Delta_2=\{\pm 1,\pm 7\}$ & $3$
 \\ \hline
  $25$ & $\Delta_1=\{\pm 1,\pm 7\}$ & $4$
 \\ \hline
  $26$ & $\Delta_1=\{\pm 1,\pm 5\}$ & $4$
 \\ \hline
 $26$ & $\Delta_2=\{\pm 1,\pm 3,\pm 9\}$ & $4$
 \\ \hline
  $28$ & $\Delta_1=\{\pm 1,\pm 13\}$ & $4$
 \\ \hline
 $28$ & $\Delta_2=\{\pm 1,\pm 3,\pm 9\}$ & $4$
 \\ \hline
 $29$ & $\Delta_2=\{\pm 1,\pm 4,\pm 5,\pm 6,\pm 7,\pm 9,\pm 13\}$ & $4$
 \\ \hline
  $36$ & $\Delta_2=\{\pm 1,\pm 11,\pm 13\}$ & $3$
 \\ \hline
  $37$ & $\Delta_3=\{\pm 1,\pm 6,\pm 8,\pm 10,\pm 11,\pm 14\}$ & $4$
 \\ \hline
 $37$ & $\Delta_4=\{\pm 1,\pm 3,\pm 4,\pm 7,\pm 9,\pm 10,\pm 11,\pm 12,\pm 16\} $ & $4$
 \\ \hline
  $49$ & $\Delta_2=\{\pm 1,\pm 6,\pm 8,\pm 13,\pm 15,\pm 20,\pm 22\}$ & $3$
 \\ \hline
 $50$ & $\Delta_2=\{\pm 1,\pm 9,\pm 11,\pm 19,\pm 21\}$ & $4$
 \\ \hline
\end{longtable}
\begin{longtable}{|c|c|c|}
\caption{$X_\Delta(N)$ with $\mathrm{Gon}(X_\Delta(N))> 3$ and bielliptic}\label{Gon >3 table and bielliptic}
\\ \hline
$N$ & $\{\pm 1\} \subsetneq \Delta \subsetneq (\Z/N\Z)^\times$ &
$g_{X_\Delta(N)}$
 \\ \hline
  $32$ & $\Delta_1=\{\pm 1,\pm 15\}$ & $5$
 \\ \hline
  $34$ & $\Delta_2=\{\pm 1,\pm 9,\pm 13,\pm 15\}$ & $5$
 \\ \hline
  $45$ & $\Delta_4=\{\pm 1,\pm 4,\pm 11,\pm 14,\pm 16,\pm 19\}$ & $5$
 \\ \hline
  $64$ & $\Delta_3=\{\pm 1,\pm 7,\pm 9,\pm 15,\pm 17,\pm 23,\pm 25,\pm 31\}$ & $5$
 \\ \hline
 \end{longtable}
\begin{center}
\begin{longtable}{|c|c|c|}
\caption{$X_\Delta(N)$ with $\mathrm{Gon}(X_\Delta(N))> 3$ and not bielliptic}\label{Gon >3 table}
\\ \hline
$N$ & $\{\pm 1\} \subsetneq \Delta \subsetneq (\Z/N\Z)^\times$ &
$g_{X_\Delta(N)}$
 \\ \hline
 $29$ & $\Delta_1=\{\pm 1,\pm 12\}$ & $8$
 \\ \hline
 $31$ & $\Delta_1=\{\pm 1,\pm 5,\pm 6\}$ & $6$
 \\ \hline
 $31$ & $\Delta_2=\{\pm 1,\pm 2,\pm 4,\pm 8,\pm 15\}$ & $6$
 \\ \hline
 $34$ & $\Delta_1=\{\pm 1,\pm 13\}$ & $9$
 \\ \hline
 $36$ & $\Delta_1=\{\pm 1,\pm 17\}$ & $7$
 \\ \hline
 $37$ & $\Delta_1=\{\pm 1,\pm 6\}$ & $16$
 \\ \hline
 $37$ & $\Delta_2=\{\pm 1,\pm 10,\pm 11\}$ & $10$
 \\ \hline
 $43$ & $\Delta_1=\{\pm 1,\pm 6,\pm 7\}$ & $15$
 \\ \hline
 $43$ & $\Delta_2=\{\pm 1,\pm 2,\pm 4,\pm 8,\pm 11,\pm 16,\pm 21\}$ & $9$
 \\ \hline
 $45$ & $\Delta_1=\{\pm 1,\pm 19\}$ & $21$
 \\ \hline
 $45$ & $\Delta_2=\{\pm 1,\pm 14,\pm 16\}$ & $9$
 \\ \hline
 $45$ & $\Delta_3=\{\pm 1,\pm 8,\pm 17,\pm 19\}$ & $11$
 \\ \hline
  $49$ & $\Delta_1=\{\pm 1,\pm 18,\pm 19\}$ & $19$
 \\ \hline
 $50$ & $\Delta_1=\{\pm 1,\pm 7\}$ & $22$
 \\ \hline
 $54$ & $\Delta_1=\{\pm 1,\pm 17,\pm 19\}$ & $10$
 \\ \hline
  $64$ & $\Delta_1=\{\pm 1,\pm 31\}$ & $37$
 \\ \hline
 $64$ & $\Delta_2=\{\pm 1,\pm 15,\pm 17,\pm 31\}$ & $13$
 \\ \hline
  $81$ & $\Delta_1=\{\pm 1,\pm 26,\pm 28\}$ & $46$
 \\ \hline
 $81$ & $\Delta_2=\{\pm 1,\pm 8,\pm 10,\pm 17,\pm 19,\pm 26,\pm 28,\pm 35,\pm 37\}$ & $10$
 \\ \hline
\end{longtable}
\end{center}


\begin{center}
\begin{longtable}{|c|c|}
\caption{Models and Quadratic surface for $X_\Delta(N)$ with $g_{X_\Delta(N)}=4$}\label{quadratic surface for genus 4 curves}
\\ \hline
Curve & Petri's model and Quadratic surface
\\ \hline
$X_{\Delta_1}(25)$ &$
\begin{cases}
x^2 w - x w^2 - y^3 + y^2 z - 3 y z w + 3 y w^2 + z^2 w - 2 z w^2 + w^3,\\
x z - y^2 + y w - 2 z w + w^2
\end{cases}
$\\
& Diagonal form: $x^2 + y^2 - z^2 - 5w^2$, lies on a ruled surface over $\Q(\sqrt{5})$
\\ \hline
$X_{\Delta_1}(26)$ &$
\begin{cases}
x^2 z - x y^2 - x z^2 + 2 y^2 z - 2 y z^2 + y z w - y w^2 + z^3 - 2 z^2 w + z w^2,\\
x w - y z + z^2
\end{cases}
$\\
& Diagonal form: $x^2 + y^2 - z^2 - w^2$, lies on a ruled surface over $\Q$
\\ \hline
$X_{\Delta_2}(26)$ &$
\begin{cases}
x^2 z - x y^2 - x z^2 + 3 y^2 z - 6 y z^2 - 2 y z w - y w^2 + 4 z^3 + z^2 w +z w^2,\\
x w - y z + z^2 - z w
\end{cases}
$\\
& Diagonal form: $x^2 + 5y^2 - 5z^2 - w^2$, lies on a ruled surface over $\Q$
\\ \hline
$X_{\Delta_1}(28)$ & $
\begin{cases}
x^2 w - 4 x w^2 - y^2 z - 2 z^2 w + 5 z w^2 + w^3,\\
x z - x w - y^2 - z^2 - z w + 3 w^2
\end{cases}
$\\
& Diagonal form: $x^2 - y^2 - z^2 + w^2$, lies on a ruled surface over $\Q$
\\ \hline
$X_{\Delta_2}(28)$ & $
\begin{cases}
x^2 w - x y w + 3 x w^2 - y z^2 + 4 y z w - 4 y w^2 - 6 z w^2 + 2 w^3,\\
x z + x w - y^2 + y z + y w - z^2 - 2 z w - w^2
\end{cases}
$\\
& Diagonal form: $3x^2 - 3y^2 - z^2$, lies on a quadratic cone over $\Q$
\\ \hline
$X_{\Delta_2}(29)$ & $
\begin{cases}
x^2 z - 3 x y z - x z^2 - y^3 + y^2 z + y z^2 + 5 y z w - y w^2 + 4 z^3 -5 z^2 w - 3 z w^2,\\
x w - y z + y w + z^2 - 3 z w - 2 w^2
\end{cases}
$\\
& Diagonal form: $4x^2 + 2y^2 - 4z^2 - 2w^2$, lies on a ruled surface over $\Q$
\\ \hline
$X_{\Delta_3}(37)$ & $
\begin{cases}
 x^2 z - x y^2 - x y z + y^3 - y^2 z + 3 y^2 w - y z^2 + y z w + y w^2 + z^3 -5 z^2 w + 8 z w^2 - 4 w^3,\\
x w - y z + y w - z^2 + 2 z w - w^2
\end{cases}
$\\
&Diagonal form: $x^2 + y^2 - z^2 - w^2$, lies on a ruled surface over $\Q$
\\ \hline
$X_{\Delta_4}(37)$ & $
\begin{cases}
2 x^2 w - 5 x w^2 - 2 y^3 + 2 y^2 z - 2 y z^2 + 6 y z w - 6 y w^2 - 3 z^3 +8 z^2 w - 9 z w^2 + 6 w^3,\\
x^2 + y^2 - z^2
\end{cases}
$\\
&Diagonal form: $x^2 + y^2 - z^2$, lies on a quadratic cone over $\Q$
\\ \hline
$X_{\Delta_2}(50)$ & $
\begin{cases}
x^2 z - x y^2 - y w^2 - z^2 w,\\
x w - y z
\end{cases}
$\\
& Diagonal Form: $x w - y z$, lies on a ruled surface over $\Q$
\\ \hline
\end{longtable}
\end{center}

\bibliographystyle{alpha}

\noindent{Tarun Dalal}\\
 {Department of Mathematics, Harish-Chandra Research Institute, HBNI\\ Chhatnag Road, Jhunsi, Prayagraj 211019, India}\\
 {tarun.dalal80@gmail.com}

 \end{document}